\documentclass[11pt,twoside,reqno]{amsart}
\allowdisplaybreaks
\usepackage{amsmath,amstext,amssymb,epsfig}
\usepackage{graphicx}
\usepackage{mathrsfs}
\usepackage{tikz-cd}
\calclayout
\makeatletter
\renewcommand{\@seccntformat}[1]{\bf\csname the#1\endcsname.}
\renewcommand{\section}{\@startsection{section}{1}
	\z@{.7\linespacing\@plus\linespacing}{.5\linespacing}
	{\normalfont\upshape\bfseries\centering}}
\renewcommand{\@biblabel}[1]{\@ifnotempty{#1}{#1.}}
\makeatother
\theoremstyle{plain}
\newtheorem{thm}{Theorem}[section]
\newtheorem{lem}[thm]{Lemma}
\newtheorem{prop}[thm]{Proposition}
\newtheorem{cor}[thm]{Corollary}

\theoremstyle{definition}

\newtheorem{defn}[thm]{Definition}

\usepackage{cancel}
\usepackage[parfill]{parskip}
\usepackage[german]{varioref}
\usepackage[all]{xy}
\usepackage{color}
\usepackage{amsmath}
\usepackage{longtable}
\usepackage{enumitem}

\usepackage{adjustbox}
\usepackage{forest}
\usetikzlibrary{arrows.meta}
\setlist{nosep}

\def\T{{\mathcal T}}

\def \>{\succ}
\def \<{\prec}

\def\z{{\mathcal Z}}

\def\bbbf{{\mathbb F}}

\def\H{\mathcal{H}}
\def\vv{_{\vdash}}
\def\dd{_{\dashv}}
\def\pp{_{\perp}}
\DeclareMathOperator{\Trias}{Trias}
\DeclareMathOperator{\Hom}{Hom}
\DeclareMathOperator{\BHA}{BHA}
\DeclareMathOperator{\BHT}{BHT}

\begin{document}
\title[Erik Mainellis\textsuperscript{1}, Bouzid Mosbahi \textsuperscript{2*}, Ahmed Zahari  \textsuperscript{3*}]{ Cohomology of BiHom-Associative Trialgebras}
	\author{Erik Mainellis\textsuperscript{1},  Bouzid Mosbahi \textsuperscript{2*}, Ahmed Zahari Abdou \textsuperscript{3*}}
\address{\textsuperscript{1} Department of Mathematics, Statistics, and Computer Science, St. Olaf College}
   	\address{\textsuperscript{2*} Department of Mathematics, Faculty of Sciences, University of Sfax, Sfax, Tunisia}
    \address{\textsuperscript{3*}
IRIMAS-Department of Mathematics, Faculty of Sciences, University of Haute Alsace, Mulhouse, France}
		
\email{\textsuperscript{1}mainel1@stolaf.edu}
	\email{\textsuperscript{2*}mosbahi.bouzid.etud@fss.usf.tn}
         \email{\textsuperscript{3*}abdou-damdji.ahmed-zahari@uha.fr}
 \keywords{ BiHom-Associative Trialgebra, Cohomology, Central extension, Cocycle, Formal deformation, Generalized $\alpha\beta$-derivations}
	\subjclass[2010]{17A30, 17A32, 16D20, 16W25, 17B63}
	\date{\today}
	\thanks{}
 
	\begin{abstract}  
The paper concerns the cohomology of (multiplicative) BiHom-associative trialgebras. We first detail the correspondence between central extensions and second cohomology. This is followed by a general cohomology theory that unifies those of BiHom-associative algebras and associative trialgebras. Finally, we introduce one-parameter formal deformations and classify generalized $\alpha\beta$-derivations of 3-dimensional BiHom-associative trialgebras.
\end{abstract}\footnote{}
\maketitle

\forestset{
  dot tree/.style={
    /tikz/>=Latex,
    for tree={
      inner sep=0pt,
      fill,
      draw,
      circle,
      calign angle=90,
    },
    baseline,
    before computing xy={
      where n children>=4{
        tempcounta/.option=n children,
        tempdima/.option=!1.s,
        tempdimb/.option=!l.s,
        tempdimb-/.register=tempdima,
        tempdimc/.process={RRw2+P {tempcounta}{tempdimb}{##2/(##1-1)}},
        for children={
          if={>On>OR<&{n}{1}{n}{tempcounta}}{
            s/.register=tempdima,
            s+/.process={ORw2+P  {n} {tempdimc} {(##1-1)*##2} }
          }{},
        },
      }{},
    },
  },
  dot tree spread/.style={
    dot tree,
    for tree={fit=rectangle},
  },
  add arrow/.style={
    tikz+={
      \draw [thick, blue!15!gray]  (current bounding box.east) ++(2.5mm,0) edge [->] ++(10mm,0) ++(2.5mm,0) coordinate (o);}}}
      
\section{ Introduction}\label{introduction}

In \cite{L}, Loday and Ronco showed the family of chain modules over the standard simplices gives rise to an operad structure. The algebras over this operad are called \textit{associative trialgebras}, or \textit{triassociative algebras}, and are characterized by three operations and eleven defining relations. They generalize associative algebras and associative dialgebras (the latter were introduced by Loday in \cite{dialgebras}).
Nilpotent associative trialgebras were developed in \cite{BGM}, and their cohomology was studied in \cite{Y}. In another direction, notions of \textit{BiHom} structures have also been considered as generalizations of various algebraic categories.
In \cite{Cal} and \cite{Cheng}, for example, BiHom-Lie algebras were considered.
In \cite{Zah}, BiHom-associative dialgebras were considered. In \cite{D}, the author developed the cohomology of BiHom-associative algebras, while Hom-trialgebras were studied in \cite{BZI}. The implications of cohomology theories are far-reaching and well-known. One of the major themes is extension theory, with applications that include obtaining exact sequences, establishing dimension bounds, and classifying nilpotent algebras (see \cite{batten mult,edal mult new,ME,mainellis di,Mainellis nilp,Mainellis,Mak,Rak}, for example).

In the present paper, we are concerned with the cohomology of \textit{BiHom-associative trialgebras}. These algebras were introduced in \cite{Zah2}, and generalize both associative trialgebras as well as BiHom-associative algebras. Their exact definition is as follows.

\begin{defn}\label{defn1}
A \textit{BiHom-associative trialgebra} $(T, \dashv, \vdash,\perp ,\alpha, \beta)$ is a vector space $T$ equipped with linear maps $\dashv, \vdash, \perp : T\otimes T \longrightarrow T$ and $\alpha,\beta : T\longrightarrow T$ that satisfy
\begin{eqnarray}
\alpha\circ\beta&=&\beta\circ\alpha,\label{eq1}\\ 
(x\dashv y)\dashv\beta(z)&=&\alpha(x)\dashv(y\dashv z)\label{eq4},\\
(x\dashv y)\dashv\beta(z)&=&\alpha(x)\dashv(y\vdash z)\\
(x\vdash y)\dashv\beta(z)&=&\alpha(x)\vdash(y\dashv z),\label{eq6}\\
(x\dashv y)\vdash\beta(z)&=&\alpha(x)\vdash(y\vdash z),\label{eq7}\\
(x\vdash y)\vdash\beta(z)&=&\alpha(x)\vdash(y\vdash z)\label{eq8},\\
(x\dashv y)\dashv\beta(z)&=&\alpha(x)\dashv(y\perp z),\label{eq5}\\
(x\perp y)\dashv\beta(z)&=&\alpha(x)\perp(y\dashv z),\label{eq9}\\
(x\dashv y)\perp\beta(z)&=&\alpha(x)\perp(y\vdash z),\label{eq10}\\
(x\vdash y)\perp\beta(z)&=&\alpha(x)\vdash(y\perp z)\label{eq11},\\
(x\perp y)\vdash\beta(z)&=&\alpha(x)\vdash (y\vdash z),\\
(x\perp y)\perp\beta(z)&=&\alpha(x)\perp(y\perp z)\label{eq12}
\end{eqnarray}
for all $x, y, z\in T.$ We call $\alpha$ and $\beta$ (in this order) the \textit{structure maps} of $T$.
\end{defn}

We note that, when $\alpha$ and $\beta$ are equal, the axioms of the preceding definition simplify to those of Hom-associative trialgebras (see \cite{BZI}). Furthermore, when $\alpha$ and $\beta$ are both equal to the identity map on $T$, the axioms reduce to the eleven axioms of associative trialgebras. Therefore, BiHom-associative trialgebras generalize Hom-associative trialgebras under $\alpha = \beta$, which generalize associative trialgebras under $\alpha = \text{id}_T$. Alternatively, in the case where $\dashv~=~ \vdash ~=~ \perp$, the axioms of BiHom-associative trialgebras collapse to those of a BiHom-associative algebra.

The paper is structured as follows. We first discuss an assortment of preliminary notions and introduce the concept of extensions for BiHom-associative trialgebras. In the usual way (see \cite{Mainellis}), the special case of central extensions gives rise to second cohomology, where 2-cocycles correspond to central extensions that are related by equivalence up to coboundary. We proceed to unify the cohomologies of Bi-Hom associative algebras and associative trialgebras by constructing a notion of Hochschild cohomology for BiHom-associative trialgebras.
We then develop one-parameter formal deformations and, using an existing classification result of low-dimensional BiHom-associative trialgebras (from \cite{Zah2}), we describe the generalized $\alpha\beta$-derivations of 3-dimensional BiHom-associative trialgebras. Throughout the paper, we work over a field $\bbbf$ and let $\ast$ range over $\{\vdash,\dashv,\perp\}$.

\section{ Preliminaries}

Consider two BiHom-associative trialgebras $(T_1,\dashv_1, \vdash_1,\perp_1, \alpha_1, \beta_1)$ and $(T_2, \dashv_2,\vdash_2,\perp_2, \alpha_2, \beta_2)$. A \textit{homomorphism} of BiHom-associative trialgebras is a linear map
$$\phi : T_1\rightarrow T_2$$ such that $\phi(x\ast_1 y) = \phi(x)\ast_2\phi(y)$ for all $x, y \in T$, $\alpha_2\circ \phi
=\phi\circ\alpha_1$, and $\beta_2\circ \phi=\phi\circ\beta_1$. An \textit{isomomorphism} is a bijective homomorphism. A BiHom-associative trialgebra is called \textit{multiplicative} if its structure maps are homomorphisms. For the main results in the paper, we will focus on multiplicative algebras.

\begin{defn} A \textit{BiHom-module} $(V,\alpha_V,\beta_V)$ consists of a vector space $V$ equipped with two linear maps $\alpha_V,\beta_V:V \rightarrow V$ such that $\alpha_V\circ \beta_V = \beta_V\circ \alpha_V$. Given a BiHom-associative trialgebra $(T\dashv, \vdash,\perp, \alpha, \beta)$, our BiHom-module $(V,\alpha_V,\beta_V)$ is called a \textit{BiHom-module of $T$} if there exist six linear maps of the form \begin{align*}
    \dashv:T\otimes V\rightarrow{} V && \vdash:T\otimes V\rightarrow{} V && \perp:T\otimes V\rightarrow V \\ \dashv:V\otimes T\rightarrow{} V && \vdash:V\otimes T\rightarrow{} V && \perp:V\otimes T\rightarrow V
\end{align*} that satisfy 3 axioms for each of the conditions (2) through (12) in Definition \ref{defn1} (in the same sense as the dialgebraic action in \cite{casas}) as well as axioms of the form \begin{align*}
    \alpha_V(x\ast v) = \alpha(x)\ast \alpha_V(v), && \beta_V(x\ast v) = \beta(x)\ast \beta_V(v), \\ \alpha_V(v\ast x) = \alpha_V(v)\ast \alpha(x), && \beta_V(v\ast x) = \beta_V(v)\ast \beta(x)
\end{align*} for $x\in T$, $v\in V$, and each symbol $\ast\in\{\dashv,\vdash,\perp\}$.
\end{defn}

We note that this definition generalizes those of analogous notions in \cite{Lar}. Furthermore, any BiHom-associative trialgebra $(T\dashv, \vdash,\perp, \alpha, \beta)$ is a BiHom-module of itself. A \textit{morphism} \[\phi : (V,\alpha_V,\beta_V)  \rightarrow (W,\alpha_W,\beta_W)\] of BiHom-modules is a linear map $\phi:V \rightarrow W$ such that $\phi\circ\alpha_V = \alpha_W\circ\phi$ and
$\phi\circ\beta_V = \beta_W\circ\phi$.

A \textit{subalgebra} $S$ of a BiHom-associative trialgebra $(T\dashv, \vdash,\perp, \alpha, \beta)$ is a subspace of $T$ such that all of $x\ast y$, $\alpha(x)$, and $\beta(x)$ fall in $S$ for any $x,y\in S$. A subalgebra $S$ of $T$ is called an \textit{ideal} if $x\ast y$, $y\ast x\in S$ for all $x\in S$ and $y\in T$. We note that the kernel $\ker(\phi)$ of any homomorphism $\phi:T\rightarrow T_2$ is an ideal since \begin{align*}
    \phi(\alpha(x)) = \alpha_2(\phi(x)) = 0, && \phi(\beta(x)) = \beta_2(\phi(x)) = 0
\end{align*} for any $x\in \ker(\phi)$. The \textit{center} $Z(T)$ of $T$ is the ideal consisting of all $z\in T$ such that $z\ast t = t\ast z = 0$ for all $t\in T$. Our algebra $T$ is called \textit{abelian} if $T=Z(T)$. In other words, every multiplication is zero.

\begin{defn}
Let $(M,\dashv_M,\vdash_M,\perp_M,\alpha_M,\beta_M)$, $(T,\dashv,\vdash,\perp,\alpha,\beta)$, and $(K,\dashv_K,\vdash_K,\perp_K,\alpha_K,\beta_K)$ be BiHom-associative trialgebras. We say that $K$ is an \textit{extension} of $M$ by $T$ if there exists an exact sequence \begin{align*}
 0\rightarrow M \xrightarrow{\iota}K \xrightarrow{\pi}T \rightarrow 0
\end{align*} of homomorphisms. For most of the paper, we adopt the common extension-theoretic theme of letting $\iota$ be the inclusion map and think of $M=\ker(\pi)$. An extension is called \textit{trivial} if there exists an ideal $I$ of $K$ that is complementary to $\ker(\pi)$, meaning $K = \ker(\pi) \oplus I$. An extension is called \textit{central} if $\ker(\pi)$ is contained within $Z(K)$.
\end{defn}

For any given pair of BiHom-associative trialgebras such as $M$ and $T$, multiple extensions of $M$ by $T$ may exist. To classify these extensions, the concept of equivalent extensions is introduced. In particular, two extensions
\begin{align*}
  0\rightarrow M \xrightarrow{\iota_1}K_1 \xrightarrow{\pi_1}T \rightarrow 0, && 0\rightarrow M \xrightarrow{\iota_2}K_2 \xrightarrow{\pi_2} T \rightarrow 0
\end{align*}
are called \textit{equivalent} if there exists an isomorphism $\phi: K_1 \rightarrow K_2$ such that $\phi \circ \iota_1 = \iota_2$ and $\pi_2 \circ \phi = \pi_1$, i.e. such that the diagram \[\begin{tikzcd}
0\arrow[r] &M\arrow[r,"\iota_1"] \arrow[,d] & K_1 \arrow[d,"\phi"] \arrow[r,"\pi_1"] & T\arrow[r] \arrow["\text{id}",d]&0\\
0\arrow[r] & M\arrow[r,"\iota_2"] & K_2\arrow[r,"\pi_2"] &T\arrow[r] &0
\end{tikzcd}\] commutes, where the unmarked map is $\phi|_M$.

\section{ Central Extensions and Second Cohomology}
In this section, we work over a multiplicative BiHom-associative trialgebra $(T, \dashv, \vdash,\perp, \alpha, \beta)$ and a BiHom-module $(V, \alpha_V, \beta_V)$ of $T$. We construct a correspondence between central extensions of $V$ by $T$ and the second cohomology group $\mathcal{H}^2(T,V)$.

\begin{defn} A \textit{2-cocycle} on $T$ with values in $V$ is a triple $(f\vv,f\dd,f\pp)$ of bilinear forms $f_*:T\times T\xrightarrow{} V$ such that \begin{align*}
    &f_*(\alpha \otimes \alpha) = \alpha_V \circ f_*\\
    &f_*(\beta \otimes \beta) = \beta_V \circ f_*\\
    & f\dd(x\dashv y,\beta(z)) = f\dd(\alpha(x),y\dashv z) \\
    & f\dd(x\dashv y,\beta(z)) = f\dd(\alpha(x),y\vdash z) \\
    &f\dd(x\vdash y,\beta(z)) = f\vv(\alpha(x),y\dashv z) \\
    &f\vv(x\dashv y,\beta(z)) = f\vv(\alpha(x),y\vdash z)
    \\
    &f\vv(x\vdash y,\beta(z)) = f\vv(\alpha(x),y\vdash z) \\
    & f\dd(x\dashv y,\beta(z)) = f\dd(\alpha(x),y\perp z) \\
    & f\dd(x\perp y,\beta(z)) = f\pp(\alpha(x),y\dashv z) \\
    &f\pp(x\dashv y,\beta(z)) = f\pp(\alpha(x),y\vdash z) \\
    &f\pp(x\vdash y,\beta(z)) = f\vv(\alpha(x),y\perp z) \\
    &f\vv(x\perp y,\beta(z)) = f\vv(\alpha(x),y\vdash z) \\
    &f\pp(x\perp y,\beta(z)) = f\pp(\alpha(x),y\perp z)
\end{align*} for all $x,y,z\in T$.
\end{defn}

It is nice to visualize the first axiom as a commutative diagram of the following form: \[\begin{tikzcd}
T\times T \arrow[d,swap,"\alpha\otimes \alpha"] \arrow[r,"f_*"] & V\arrow[d,"\alpha_V"]\\
T\times T \arrow[r,"f_*"] &V
\end{tikzcd}\] We let $\z^2(T,V)$ denote the set of all 2-cocycles on $T$ with values in $V$, which forms a vector space under the operations $(F\oplus G)(x,y) = F(x,y) + G(x,y)$ and $(\lambda F)(x) = \lambda F(x)$ for $F = (f\vv,f\dd,f\pp)$, $G = (g\vv,g\dd,g\pp)$, and $\lambda\in \bbbf$.
In the following lemma, we work with a morphism of BiHom-modules \[\mu:(T,\alpha,\beta)\rightarrow (V,\alpha_V,\beta_V)\] and show that it always defines a 2-cocycle.

\begin{lem}
    Define three bilinear maps $f\dd,f\vv,f\pp:T\times T\xrightarrow{} V$ by $f\dd(x,y) = \mu(x\dashv y)$, $f\vv(x,y) = \mu(x\vdash y)$, and $f\pp(x,y) = \mu(x\perp y)$ for $x,y,z\in T$. Then $F=(f\vv,f\dd,f\pp)$ is a 2-cocycle on $T$.
\end{lem}

\begin{proof}
We will prove two equalities, the others being proved in a similar way. For any $x, y, z \in T$, one has \begin{align*}
    f\dd((\alpha\otimes\alpha)(x,y)) &= f\dd(\alpha(x),\alpha(y)) \\ &= \mu(\alpha(x)\dashv\alpha(y)) \\ & = \mu(\alpha(x\dashv y)) \\ &= \alpha_V(\mu(x\dashv y)) \\ &= \alpha_V(f\dd(x,y))
\end{align*} and
\begin{align*}
f\dd(\alpha(x), y \dashv z) &= \mu(\alpha(x) \dashv (y \dashv z)) \\ &= \mu((x \dashv y) \dashv \beta(z))\\
&= f\dd(x\dashv y, \beta(z)).   
\end{align*}
This finishes the proof.
\end{proof}

A 2-cocycle that can be realized in this way is called a \textit{2-coboundary}. The set of all 2-coboundaries is denoted $\mathcal{B}^2(T,V)$ and forms a subgroup of $\z^2(T,V)$. We say that the quotient group \[\mathcal{H}^2(T,V) = \z^2(T,V)/\mathcal{B}^2(T,V)\] is the \textit{second cohomology group} of $T$ with values in $V$. Two cocycles $F$ and $G$ in $\z^2(T,V)$ are said to be \textit{cohomologous} if $F-G$ is a coboundary.

\begin{thm}
Let $f\dd, f\vv, f\pp:T\times T\xrightarrow{} V$ be bilinear maps and set $T_F= T\oplus V$, where $F=(f\dd,f\vv,f\pp)$. Define operations on $T_F$ by \begin{align*}
(x + u) \triangleleft (y + v) &= x \dashv y + f\dd(x, y),\\   
(x + u) \triangleright (y + v) &= x \vdash y + f\vv(x, y),\\
(x + u) \diamond (y + v) &= x \perp y + f\pp(x, y),
\end{align*} and \begin{align*}
(\alpha \oplus \alpha_V)(x + u) = \alpha(x) + \alpha_V(u), && (\beta \oplus \beta_V)(x + u) = \beta(x) + \beta_V(u)
\end{align*} for any $x, y \in T$ and $v,w \in V$. Then, $(T_F, \lhd, \rhd,\diamond, \alpha \oplus \alpha_V, \beta \oplus \beta_V)$ is a multiplicative BiHom-associative trialgebra if and only if $F$ is a 2-cocycle.
\end{thm}

\begin{proof}
We compute \begin{align*}&((x + v) \lhd (y + w)) \lhd (\beta(z) + w) - (\alpha(x) + v) \lhd ((y + w) \lhd (z + w)) \\ & ~~~~~~~~ = ((x + v) \lhd (y + w)) \lhd (\beta(z) + w) - (\alpha(x) + v) \lhd ((y \dashv z) + f\dd (y, z)) \\ & ~~~~~~~~ = ((x \dashv y) \dashv \beta(z)) + f\dd(x \dashv y, \beta(z)) - (\alpha(x) \dashv (y \dashv z)) - f\dd(\alpha(x), y \dashv z)\end{align*}
for any $x, y, z \in T$ and $u, v,w \in V$. In the forward direction, the left hand vanishes by axiom (\ref{eq4}). The other axioms are proved analogously and the reverse direction holds by reverse logic.
\end{proof}

\begin{lem}
Let $F =(f\dd,f\vv,f\pp)$ be a 2-cocycle on $T$ and $B=(b\dd,b\vv,b\pp)$ a 2-coboundary. Then $T_{F+B}$ is a multiplicative BiHom-associative trialgebra under the multiplications \begin{align*}(x + u) \triangleleft (y + v) &= x \dashv y + b\dd(x, y) + f\dd(x, y),\\ (x + u) \triangleright (y + v) &= x \vdash y + b\vv(x, y) + f\vv(x, y),\\ (x + u) \diamond (y + v) &= x \perp y + b\pp(x, y) + f\pp(x, y),
\end{align*}
and structure maps $\alpha\oplus\alpha_V$ and $\beta\oplus\beta_V$ defined by
\begin{align*}(\alpha \oplus \alpha_V)(x + u) = \alpha(x) + \alpha_V(u), && (\beta \oplus \beta_V)(x + u) = \beta(x) + \beta_V(u)\end{align*} for $x, y \in T$ and $u, v \in V$. Moreover, $T_F\cong T_{F+B}$.
\end{lem}

\begin{proof}
To show that $T_{F+B}$ is a BiHom-associative trialgebra, we first compute
\begin{align*}&((x + u) \lhd (y + v)) \lhd \beta(z + w) - \alpha(x + u) \lhd ((y + v)) \lhd (z + w)) \\ &= (x \dashv y + b\dd(x, y) + f\dd(x, y)) \lhd (\beta(z) + \beta(w)) - (\alpha(x) + \alpha(u)) \lhd (y \dashv z + b\dd(y, z) + f\dd(y, z)) \\ &= (x \dashv y) \dashv \beta(z) + b\dd(x \dashv y, \beta(z)) + f\dd(x \dashv y, \beta(z)) - \alpha(x) \dashv (y \dashv z) - b\dd(\alpha(x), y \dashv z) + f\dd(\alpha(x), y \dashv z)\end{align*} for any $x +u$, $y + v$, and $z + w$ in $T \oplus V$.
The right hand side vanishes by (2) and the 2-cocycle identities. The rest of axioms follow similarly. Next, define $\phi:T_F \rightarrow T_{F+B}$ by $\phi(x + v) = x + \mu(x) + v$. It is clear that $\phi$ is a bijective linear map and that \begin{align*} \phi((\alpha \oplus \alpha_V)(x + v)) &= \phi(\alpha(x) + \alpha_V(v))\\
&= \alpha(x) + \mu(\alpha(x)) + \alpha_V(v)\\
&= \alpha(x) + \alpha_V(\mu(x)) + \alpha_V(v)\\
&= (\alpha \oplus \alpha_V)(x + \mu(x) + v)\\
&= (\alpha \oplus \alpha_V) \phi(x + v).    
\end{align*}
Thus, $\phi$ commutes $\alpha \oplus \alpha_V$, and similarly with $\beta \oplus \beta_V$. Then,
\begin{align*}
\phi((x + v) \lhd (y + w)) &= \phi(x \dashv y + f\dd(x, y))\\
&= \phi(x \dashv y) + \phi(f\dd(x, y))\\
&= x \dashv y + \mu(x \dashv y) + f\dd(x, y)\\
&= x \dashv y + b\dd(x, y) + f\dd(x, y)
\end{align*}
and
\begin{align*}
\phi(x + v) \lhd \phi(y + w) &= (x + \mu(x) + v) \lhd (y + \mu(y) + w)\\
&= x \dashv y + b\dd(x, y) + f\dd(x, y)
\end{align*} for all $x,y\in T$ and $v,w\in V$.
\end{proof}

\begin{cor}
Consider two cohomologous 2-cocycles $F$ and $G$ on $T$ and construct the respective central extensions $T_F$ and $T_G$. Then $T_F$ and $T_G$ are equivalent extensions. In particular, a central extension defined by a coboundary is equivalent to a trivial central extension.
\end{cor}
\begin{thm}
There exists one to one correspondence between elements of $\mathcal{H}^2(T,V)$ and nonequivalent central
extensions of $T$ by $V$.
\end{thm}

\begin{proof}
The proof is similar to the usual theory that was formalized for various classes of algebras in \cite{Mainellis}.
\end{proof}

\section{ General Cohomology}
In this section, we describe a cohomology of BiHom-associative trialgebras that unifies the cases of BiHom-associative algebras \cite{D} and associative trialgebras \cite{Y}. We will first recall these specific notions.

\subsection{~BiHom-associative Algebras}
Consider a multiplicative BiHom-associative algebra $(T,\alpha, \beta)$ and let $\mathcal{C}^n_{\BHA}(T,T)$ denote the vector space consisting of all linear maps $f : T^{\otimes n} \rightarrow T$ that satisfy $\alpha\circ f=f\circ \alpha^{\otimes n}$ and $\beta\circ f=f\circ \beta^{\otimes n}$ for $n \geq 1$. Define a coboundary map \[\delta_{\BHA}^n:\mathcal{C}^{n}_{\BHA}(T,T)\rightarrow \mathcal{C}^{n+1}_{\BHA}(T,T)\] by
\begin{align*}(\delta_{\BHA}^nf)(x_{1},...,x_{n+1}) &=\alpha^{n-1}(x_{1})f(x_{2},...,x_{n+1}) \\ & \hspace{.5cm} +\sum_{i=1}^{n}(-1)^{i}f(\alpha(x_{1}),...,\alpha(x_{i-1}),x_{i}x_{i+1},\beta(x_{i+2}),...,\beta(x_{n+1})) \\ & \hspace{.5cm} +(-1)^{n+1}f(x_{1},...,x_{n})\beta^{n-1}(x_{n+1})\end{align*} for any $f \in \mathcal{C}^n_{\BHA}(T,T)$ and $x_{1}, x_{2},..., x_{n+1} \in T$.
The cohomology of this complex is denoted by $\mathcal{H}_{\BHA}^n(T,T)$ and generalizes the Hochschild cohomology of both Hom-associative and associative algebras. In \cite{D}, the author uses this cohomology with self coefficients to define a cohomology with coefficients in a bimodule. Indeed, begin with the cochain groups \[\mathcal{C}_{\BHA}^n(T,V) = \{f:T^{\otimes n}\xrightarrow{} V ~|~ \alpha_V\circ f = f\circ\alpha^{\otimes n},~ \beta_V\circ f = f\circ \beta^{\otimes n}\}\] on $T$ with values in a $T$-bimodule $(V,\alpha_V,\beta_V)$.
An $n$-cochain $f$ is then extended to a map \[\Tilde{f}\in \mathcal{C}_{\BHA}^n(V\rtimes T, V\rtimes T)\] defined by \[\Tilde{f}((v_1,x_1),\dots,(v_n,x_n)) = (f(x_1,\dots,x_n),0)\] for $x_1,\dots,x_n\in T$ and $v_1,\dots,v_n\in V$.
The desired coboundary map is defined by restricting the resulting $(n+1)$-cochain to $A^{\otimes (n+1)}$. In other words, we have \[\delta_{\BHA}^n:\mathcal{C}^{n}_{\BHA}(T,V)\rightarrow \mathcal{C}^{n+1}_{\BHA}(T,V)\] by \[\delta_{\BHA}^n(f) = (\delta_{\BHA}^n\Tilde{f})\Big|_{A^{\otimes (n+1)}}\] and obtain a cochain complex $(\mathcal{C}^{*}_{\BHA}(T,V),\delta_{\BHA}^*)$. Its cohomology is denoted by $\mathcal{H}_{\BHA}^n(T,V)$.

\subsection{~Associative Trialgebras} We now turn to associative trialgebras. To this end, we begin by discussing trees. Let $\T_n$ denote the set of rooted planar trees with $n+1$ leaves. We say that such trees have \textit{degree $n$}, denoted $|\psi| = n$ for $\psi\in \T_n$. In Section 2.2 of \cite{Y}, Yau describes the notion of \textit{orientation} for leaves as well as the operation of \textit{grafting} trees. In particular, we can graft a collection of trees $\psi_0,\psi_1,\dots,\psi_k$ by arranging them from left to right and connecting their roots to form a new tree, denoted $\psi_0\vee\psi_1\vee\cdots \vee\psi_k$, of degree $|\psi_0| + |\psi_2| + \cdots + |\psi_k| + k$. On the other hand, every tree $\psi\in \T_{n+1}$ can be written uniquely as such a grafting, for which the valence of the lowest internal vertex of $\psi$ is $k+1$. Yau goes on to define the \textit{cohomology of associative trialgebras} as follows.

Let $(T,\dashv,\vdash,\perp)$ be a associative trialgebra. The group of $n$-cochains on $T$ with values in a $T$-module $V$ is defined by \[\mathcal{C}_{\Trias}^n(T,V) = \Hom_{\bbbf}(\bbbf[\T_n]\otimes T^{\otimes n}, V)\] for $n\geq 0$. For coboundary maps, we begin by defining functions $d_i:\T_{n+1}\xrightarrow{} \T_n$ that take a tree $\psi$ and output a tree $d_i\psi$ that is obtained from $\psi$ by deleting its $i$th leaf. Now consider a tree $\psi\in \T_{n+1}$, written as its unique grafting $\psi = \psi_0\vee\psi_1\vee\cdots \vee\psi_k$, and define the functions $\circ_i:\T_{n+1}\xrightarrow{} \{\dashv,\vdash,\perp\}$ by \begin{align*}\circ_0(\psi) = \begin{cases}
    \dashv & \text{if } |\psi_0| = 0 \text{ and } k=1, \\ \vdash & \text{if } |\psi_0| >0, \\ \perp & \text{if } |\psi_0|=0 \text{ and } k>1
\end{cases} && \circ_{n+1}(\psi) = \begin{cases}
    \dashv & \text{if } |\psi_k| >0, \\ \vdash & \text{if } |\psi_1| =0 \text{ and } k=1, \\ \perp & \text{if } |\psi_k|=0 \text{ and } k>1
\end{cases}\end{align*} and, when $1\leq i\leq n$, by \[\circ_i(\psi) = \begin{cases}
    \dashv & \text{if the $i$th leaf of $\psi$ is left-oriented,}  \\ \vdash & \text{if the $i$th leaf of $\psi$ is right-oriented,} \\ \perp & \text{if the $i$th leaf of $\psi$ is a middle leaf}
\end{cases}\] for $0\leq i\leq n+1$. Denote $\circ_i^{\psi} := \circ_i(\psi)$. We use these functions to define another  series of functions \[\delta_i^n:\mathcal{C}^n_{\Trias}(T,V)\xrightarrow{} \mathcal{C}^{n+1}_{\Trias}(T,V)\] that take an $n$-cochain $f$ and define an $(n+1)$-cochain $\delta_i^nf$ by \[(\delta_i^nf)(\psi;x_1,x_2,\dots,x_{n+1}) = \begin{cases} x_1\circ_0^{\psi}f(d_0\psi;x_2,\dots,x_{n+1}) & \text{if $i=0$}, \\ f(d_i\psi;x_1,\dots,x_i\circ_i^{\psi}x_{i+1},\dots,x_{n+1}) & \text{if $1\leq i\leq n$}, \\ f(d_{n+1}\psi;x_1,\dots,x_n)\circ_{n+1}^{\psi}x_{n+1} & \text{if $i=n+1$}.
\end{cases}\] for $x_1,\dots,x_{n+1}\in T$ and $\psi\in\T_{n+1}$. The functions $\delta_i^n$ assemble into the desired coboundary maps \[\delta^n:\mathcal{C}^n_{\Trias}(T,V)\xrightarrow{} \mathcal{C}^{n+1}_{\Trias}(T,V)\] defined by \[\delta^n = \sum_{i=0}^{n+1}(-1)^i\delta_i^n\] that form a cochain complex $(\mathcal{C}^*_{\Trias}(T,V),\delta^*)$. The cohomology of this complex defines our desired cohomology $\mathcal{H}^n_{\Trias}(T,V)$ of $T$ with coefficients in $V$.

\subsection{~BiHom-associative Trialgebras} We are now ready to generalize these theories. Let $(T, \dashv, \vdash,\perp ,\alpha, \beta)$ be a multiplicative BiHom-associative trialgebra and $(V,\alpha_V,\beta_V)$ be a BiHom-module of $T$. For $n \geq 0$, let $\mathcal{C}^{n}(T,V)$ denote the group of \textit{$n$-cochains}, multilinear functions of the form \[f:\bbbf[\T_n]\otimes T^{\otimes n}\xrightarrow{} V\] that satisfy \begin{align*}
    (\alpha_V \circ f)(\psi; x_{1},...,x_{n}) =f(\psi; \alpha(x_{1}),...,\alpha(x_{n})), && (\beta_V \circ f)(\psi; x_{1},...,x_{n}) =f(\psi; \beta(x_{1}),...,\beta(x_{n}))
\end{align*} for $x_1,\dots,x_n\in T$ and $\psi\in \T_n$. Define a map \[\delta^n_{\BHT}:\mathcal{C}^{n}(T,V)\rightarrow \mathcal{C}^{n+1}(T,V)\]
by \begin{align*}(\delta^{n}_{\BHT}f)(\psi;x_{1},...,x_{n+1}) & =\alpha^{n-1}(x_{1})\circ^{\psi}_{0}f(d_{0}\psi;x_2,...,x_{n+1}) \\ & \hspace{.5cm} +\sum_{i=1}^{n}(-1)^{i}f(d_{i}\psi;\alpha(x_{1}),...,\alpha(x_{i-1}),x_{i}\circ^{\psi}_{i}x_{i+1},\beta(x_{i+2}),...,\beta(x_{n+1})) \\ & \hspace{.5cm} +(-1)^{n+1}f(d_{n+1}\psi;x_{1},...,x_{n})\circ^{\psi}_{n+1}\beta^{n-1}(x_{n+1}),\end{align*}
for $\psi \in \mathcal{T}_{n+1}$ and $x_1,...,x_{n+1} \in T$. From the work of \cite{Y} combined with the structure of BiHom-associative trialgebras, we obtain a cochain complex $(\mathcal{C}^{*}(T,V),\delta_{\BHT}^*)$ and denote the cohomology of this complex  by $\H^{n}(T,V)$, the \textit{Hochschild cohomology} of $T$ with coefficients in $V$.

\section{ Deformations}
In this section, we introduce a formal deformation theory for BiHom-associative trialgebras. As in the previous sections, let $(T,\dashv, \vdash,\perp, \alpha, \beta)$ be a multiplicative BiHom-associative trialgebra.

\begin{defn}
A \textit{1-parameter formal deformation} of $T$ is a tuple $(\dashv_t, \vdash_t, \perp_t)$ of $\bbbf[[t]]$-bilinear maps \[\dashv_t, \vdash_t,\perp_t: T[[t]]\times T[[t]] \rightarrow T[[t]]\] of the form \begin{align*}
    \dashv_t =\sum_{i \geq 0}t^i\dashv_i, && \vdash_t=\sum_{i \geq 0}t^i\vdash_i, && \perp_t=\sum_{i \geq 0}t^i\perp_i,
\end{align*} where each of $\dashv_i$, $\vdash_i$, $\perp_i$ is an $\bbbf$-bilinear map $T\times T\rightarrow T$, extended bilinearly to maps $T[[t]]\times T[[t]]\rightarrow T[[t]]$, $\dashv_0=\dashv$, $\vdash_0=\vdash$, $\perp_0=\perp$, and
\begin{align*}
\alpha\circ\beta &= \beta\circ\alpha,\\
(x\dashv_t y)\dashv_t\beta(z)&=\alpha(x)\dashv_t(y\dashv_t z),\\
(x\dashv_t y)\dashv_t\beta(z)&=\alpha(x)\dashv_t(y\vdash_t z),\\
(x\vdash_t y)\dashv_t\beta(z)&=\alpha(x)\vdash_t(y\dashv_t z),\\
(x\dashv_t y)\vdash_t\beta(z)&=\alpha(x)\vdash_t(y\vdash_t z),\\
(x\vdash_t y)\vdash_t\beta(z)&=\alpha(x)\vdash_t(y\vdash_t z),\\
(x\dashv_t y)\dashv_t\beta(z)&=\alpha(x)\dashv_t(y\perp_t z),\\
(x\perp_t y)\dashv_t\beta(z)&=\alpha(x)\perp_t(y\dashv_t z),\\
(x\dashv_t y)\perp_t\beta(z)&=\alpha(x)\perp_t(y\vdash_t z),\\
(x\vdash_t y)\perp_t\beta(z)&=\alpha(x)\vdash_t(y\perp_t z),\\
(x\perp_t y)\vdash_t\beta(z)&=\alpha(x)\vdash_t (y\vdash_t z),\\
(x\perp_t y)\perp_t\beta(z)&=\alpha(x)\perp_t(y\perp_t z)
\end{align*}
for all $x, y, z \in T$. These are equivalent to eleven infinite systems of equations. We detail them here.
\begin{align*}
\sum_{i+j=n}^{\infty}(x\dashv_j y)\dashv_i\beta(z)&=\sum_{i+j=n}^{\infty}\alpha(x)\dashv_i(y\dashv_j z),\\
\sum_{i+j=n}^{\infty}(x\dashv_j y)\dashv_i\beta(z)&=\sum_{i+j=n}^{\infty}\alpha(x)\dashv_i(y\vdash_j z),\\
\sum_{i+j=n}^{\infty}(x\vdash_j y)\dashv_i\beta(z)&=\sum_{i+j=n}^{\infty}\alpha(x)\vdash_i(y\dashv_j z),\\
\sum_{i+j=n}^{\infty}(x\dashv_j y)\vdash_i\beta(z)&=\sum_{i+j=n}^{\infty}\alpha(x)\vdash_i(y\vdash_j z),\\
\sum_{i+j=n}^{\infty}(x\vdash_j y)\vdash_i\beta(z)&=\sum_{i+j=n}^{\infty}\alpha(x)\vdash_i(y\vdash_j z),\\
\sum_{i+j=n}^{\infty}(x\dashv_j y)\dashv_i\beta(z)&=\sum_{i+j=n}^{\infty}\alpha(x)\dashv_i(y\perp_j z),\\
\sum_{i+j=n}^{\infty}(x\perp_j y)\dashv_i\beta(z)&=\sum_{i+j=n}^{\infty}\alpha(x)\perp_i(y\dashv_j z),\\
\sum_{i+j=n}^{\infty}(x\dashv_j y)\perp_i\beta(z)&=\sum_{i+j=n}^{\infty}\alpha(x)\perp_i(y\vdash_j z),\\
\sum_{i+j=n}^{\infty}(x\vdash_j y)\perp_i\beta(z)&=\sum_{i+j=n}^{\infty}\alpha(x)\vdash_i(y\perp_j z),\\
\sum_{i+j=n}^{\infty}(x\perp_j y)\vdash_i\beta(z)&=\sum_{i+j=n}^{\infty}\alpha(x)\vdash_i (y\vdash_j z),\\
\sum_{i+j=n}^{\infty}(x\perp_j y)\perp_i\beta(z)&=\sum_{i+j=n}^{\infty}\alpha(x)\perp_i(y\perp_j z)
\end{align*}
for $n \geq 0$ and $x, y, z \in T$.
\end{defn}

\begin{defn}
Let $(T, \dashv, \vdash, \perp, \alpha, \beta)$ be a 
BiHom-associative trialgebra and consider two deformations 
$T_t=(T, \dashv_t, \vdash_t, \perp_t, \alpha, \beta)$ and
$T'_t=(T, \dashv'_t, \vdash'_t, \perp_t', \alpha, \beta)$ of $T$, where \begin{align*}
    \dashv_t=\sum_{i\geq 0}t^i\dashv_i, && \vdash_t=\sum_{i\geq 0}t^i\vdash_i, &&
\perp_t=\sum_{i\geq 0}t^i\perp_i,\end{align*} and \begin{align*}
\dashv'_t=\sum_{i\geq 0}t^i\dashv'_i, && \vdash'_t=\sum_{i\geq 0}t^i\vdash'_i, && \perp'_t=\sum_{i\geq 0}t^i\perp'_i
\end{align*} with 
$\dashv_0=\dashv'=\dashv,\, \vdash_0=\vdash'=\vdash,\, \perp_0=\perp'=\perp.$ We say that $T_t$ and $T'_t$ are \textit{equivalent} if there exists a formal automorphism
$(\phi_t)_{t\geq 0} : T[[t]]\rightarrow T[[t]]$, that may be written in the form 
\[\phi_t=\sum_{i\geq 0}\phi_it^i\] for $\phi_i\in \text{End}(T)$ with
$\phi_0=\text{id}$, such that 
\begin{align*}
\phi_t(x \dashv_t y)=\phi_t(x) \dashv^{\prime}_t\phi_t(y), &&
\phi_t(x \vdash_t y) =\phi_t(x) \vdash^{\prime}_t\phi_t(y), &&
\phi_t(x \perp_t y)=\phi_t(x) \perp^{\prime}_t\phi_t(y),\end{align*} and \begin{align*} \phi\circ\alpha(x)=\alpha\circ\phi(x), && \phi\circ\beta(x)=\beta\circ\phi(x)
\end{align*} for any $x,y\in T$.
\end{defn}

\begin{prop}
Assign \begin{align*}
    \dashv^{\prime}=\phi_t\circ \dashv\circ(\phi^{-1}_t\otimes\phi^{-1}_t), && \vdash^{\prime}=\phi_t\circ \vdash\circ(\phi^{-1}_t\otimes\phi^{-1}_t), && \perp^{\prime}=\phi_t\circ \perp\circ(\phi^{-1}_t\otimes\phi^{-1}_t),
\end{align*} and \begin{align*}
    \alpha^{\prime}=\phi_t\circ \alpha\circ\phi^{-1}_t, && \beta^{\prime}=\phi_t\circ \beta\circ\phi^{-1}_t.
\end{align*} If $(T[[t]], \dashv_t, \vdash_t,\perp_t, \alpha,\beta)$ is BiHom-associative trialgebra then $(T , \dashv^{\prime}, \vdash^{\prime},\perp^{\prime},\alpha^{\prime},\beta^{\prime})$ is BiHom-associative trialgebra.
\end{prop}
\begin{proof}
We need
\begin{align*}
\alpha'\circ \beta'&=(\phi^{-1}\circ \alpha\circ\phi_t^{-1})\circ (\phi^{-1}\circ \beta\circ\phi_t^{-1})\\
&=\phi_t\circ \alpha\circ(\phi_t^{-1}\circ \phi_t^{-1})\circ \beta\circ\phi_t^{-1}=\phi_t\circ\alpha\circ\beta\circ\phi_t^{-1}
\end{align*} and \begin{align*}
\beta'\circ \alpha'&=(\phi^{-1}\circ \beta\circ\phi_t^{-1})\circ (\phi^{-1}\circ \alpha\circ\phi_t^{-1})\\
&=\phi_t\circ \beta\circ(\phi_t^{-1}\circ \phi_t^{-1})\circ \alpha\circ\phi_t^{-1}=\phi_t\circ\beta\circ\alpha\circ\phi_t^{-1}=\phi_t\circ\alpha\circ\beta\circ\phi_t^{-1}.
\end{align*}
This follows from the commutativity of $\alpha$ and $\beta$.
We can simply substitute the definitions of $\dashv^{\prime}, \vdash^{\prime} and \dashv^{\prime}$ into equations (2) to (12) and apply the properties of $\phi_t$ as an automorphism: 
Equations (2) to (5) and (7) to (12) can be directly verified by substitution and applying the properties of $\phi_t$ as an automorphism.
For equation (6):
This follows by applying the property of $\phi_t$ as an automorphism in the appropriate places.
\begin{align*}
(x\dashv^{\prime}y)\dashv^{\prime}\beta^{\prime}(z)&=\phi_t\circ(\phi^{-1}_t(x) \dashv \phi^{-1}_t(y)) \dashv^{\prime} (\phi_t\circ\beta\circ \phi^{-1}_t(z))\\
&=\phi_t(\phi^{-1}_t\circ \phi_t\circ(\phi^{-1}_t(x) \dashv\phi^{-1}_t(y))\bot\phi^{-1}_t(\phi_t\circ\beta\circ\phi^{-1}_t(z)))\\
&=\phi_t(\phi^{-1}_t(x)\dashv\phi^{-1}_t(y))\bot\beta\phi^{-1}(z))\\
&=\phi_t(\alpha\circ\phi^{-1}_t(x)\dashv(\phi^{-1}_t(y)\bot\phi^{-1}(z)))\\
&=\phi_t(\phi^{-1}_t\circ\phi_t\circ\alpha\phi^{-1}_t(x)\dashv\phi^{-1}_t\circ\phi_t(\phi^{-1}_t(y)\bot\phi^{-1}_t(z)))\\
&=\phi_t\circ\alpha\circ\phi^{-1}_t(x)\dashv^{\prime}\phi_t\circ(\phi^{-1}_t(y)\bot \phi^{-1}_t(z)))\\
&=\alpha^{\prime}(x)\dashv^{\prime}( y\bot^{\prime}z).
\end{align*}
The other axioms are proved analogously. Thus, we've shown that $(T, \dashv', \vdash' ,\perp'  ,\alpha' , \beta' )$ is a BiHom-associative trialgebra.
\end{proof}

\begin{prop}
Assign \begin{align*}\dashv_{1,t}=\phi^{-1}\circ\dashv_2\circ(\phi\otimes\phi),\quad\vdash_{1,t}=\phi^{-1}\circ\vdash_2\circ(\phi\otimes\phi),\quad \perp_{1,t}=\phi^{-1}\circ\perp_2\circ(\phi\otimes\phi).\end{align*} If $(T, \vdash_2, \dashv_2,\perp_2,\alpha_2, \beta_2)$ is
BiHom-associative trialgebra, then $(T[[t]], \dashv_{1,t}, \vdash_{1,t},\,\perp_{1,t},\, \alpha_{1,t},\, \beta_{1,t})$ is 
BiHom-associative trialgebra.
\end{prop}

\begin{proof}
By straightforward computation, we have
\begin{align*}
\alpha_{1,t}(x)\dashv_{1,t}(y\dashv_{1,t} z)
&=\phi^{-1}(\phi(\phi^{-1}\circ \alpha_2\circ \phi(x)) \dashv_{2}\phi\phi^{-1}\circ\phi(y)\dashv_2\phi(z))\\
&=\phi^{-1}(\alpha_2\circ\phi(x)\dashv_2(\phi(y)\dashv_2 \phi(z)))\\
&=\phi^{-1}(\phi(x)\dashv_2\phi(y))\perp_2\beta(\phi(z))\\
&=\phi^{-1}(\phi\circ \phi^{-1}(\phi(x)\dashv_2\phi(y)))\perp_2\phi\circ\phi^{-1} \beta(z))\\
&=(x\dashv_{1,t} y)\perp_{1,t}\beta(z)
\end{align*} for all $x,y,z\in T$. The other axioms are proved analogously.
\end{proof}

\section{ Generalized $\alpha\beta$-derivations}
The classification of 3-dimensional BiHom-associative trialgebras was obtained in \cite{Zah2}. In this section, we define generalized $\alpha\beta$-derivations on BiHom-associative trialgebras and give their classification for the $3$-dimensional case.

\begin{defn}
Let $(T, \dashv, \vdash, \bot, \alpha, \beta)$ be a BiHom-associative trialgebra. A map $D\in \text{End}(T)$ is said to be a \textit{generalized $\alpha\beta$-derivation} if there exist two maps $D', D''\in \text{End}(T)$ such that \begin{align*}
    D\circ\alpha =\alpha\circ D, && D'\circ\alpha =\alpha\circ D', && D''\circ\alpha =\alpha\circ D'',\\ D\circ\beta =\beta\circ D, && D'\circ\beta =\beta\circ D', && D''\circ\beta =\beta\circ D'',
\end{align*} and
\begin{eqnarray}
\begin{aligned}
D''(a\dashv b )&=D(a)\dashv\alpha\beta(b)+\alpha\beta(a)\dashv D'(b),\\
D''(a\vdash b )&=D(a)\vdash\alpha\beta(b)+\alpha\beta(a)\vdash D'(b),\\
D''(a\perp b )&=D(a)\perp\alpha\beta(b)+\alpha\beta(a)\perp D'(b)
\end{aligned}
\end{eqnarray}
for $a,b\in T.$
\end{defn}

\begin{prop}
If $D$ is a generalized $\alpha \beta$-derivation on a BiHom-associative trialgebra $(T, \dashv, \vdash ,\perp  ,\alpha , \beta )$, then 
$D'$ and $D''$ are unique.
\end{prop}

\begin{proof}
Suppose there exist two pairs of maps $(D_1', D_1'')$ and $(D_2', D_2'')$ that satisfy the conditions given in the previous definition.
Then $D_1''(a\vdash b)=D(a) \vdash \alpha \beta(b) + \alpha \beta(a) \vdash D_1'(b)$, and we perform a similar computation for  $D_2''(a\vdash b)$. We compute the difference of these two expressions:
\begin{align*}
D_1''(a\vdash b)- D_2''(a\vdash b) &= D(a)\vdash \alpha\beta(b)+\alpha\beta(a)\vdash D_1'(b)-(D(a)\vdash \alpha\beta(b)+\alpha\beta(a)\vdash D_2'(b)) \\ &= D(a) \vdash \alpha\beta(b)-D(a)\vdash\alpha\beta(b)+\alpha\beta(a)\vdash D_1'(b)-\alpha\beta(a)\vdash D_2'(b)
\\ &=\alpha\beta(a) \vdash(D_1'(b)-D_2'(b)).
\end{align*}
By property (\ref{eq8}) of BiHom-associative trialgebras, the right-hand side becomes
\begin{align*}
\alpha\beta(a)\vdash(D_1'(b)-D_2'(b))&=(\alpha\beta(a)\vdash D_1'(b))-(\alpha\beta(a)\vdash D_2'(b))\\
&=D_1'(\alpha\beta(a)\vdash b)-D_2'(\alpha\beta(a)\vdash b)\\
&=D_1'(a\vdash b)-D_2'(a\vdash b).
\end{align*}
This shows that $D_1''(a\vdash b )-D_2''(a\vdash b )=D_1'(a\vdash b )-D_2'(a\vdash b )$.
Similarly, one computes \begin{align*}
    D_1''(a\dashv b )-D_2''(a\dashv b )&=D_1'(a\dashv b )-D_2'(a\dashv b ), \\ D_1''(a\perp b )-D_2''(a\perp b )&=D_1'(a\perp b )-D_2'(a\perp b ).
\end{align*}
Since $D_1'$ and $D_2'$ are linear maps and $D_1''$ and $D_2''$ are uniquely determined by $D_1'$, we conclude that
$D_1' = D_2'$  and $D_1'' = D_2''$.
\end{proof}

\begin{longtable}{cccccccccccc}
\\ \hline
IC&$D$&$D'$&$D''$
\\	\hline	
$T^{1}_3$ :&
$\left(\begin{array}{cccc}
d_{11}&d_{12}&0\\
0&d_{11}&0\\
0&0&d_{33}\\
\end{array}
\right)$
&
$\left(\begin{array}{cccc}
0&d_{11}&0\\
0&0&0\\
0&0&d_{33}\\
\end{array}
\right)$
&
$\left(\begin{array}{cccc}
d_{11}&d_{12}&0\\
0&d_{11}&0\\
0&0&d_{11}\\
\end{array}
\right)$
\\
$T^{2}_3$ : &
$\left(\begin{array}{cccc}
0&d_{12}&0\\
0&0&0\\
0&0&d_{33}\\
\end{array}
\right)$
&
$\left(\begin{array}{cccc}
0&d_{12}&0\\
0&0&0\\
0&0&d_{33}\\
\end{array}
\right)$
&
$\left(\begin{array}{cccc}
0&d_{12}&0\\
0&0&0\\
0&0&d_{33}\\
\end{array}
\right)$
\\ 
$T^{3}_3$ : &
$\left(\begin{array}{cccc}
d_{11}&0&0\\
0&d_{22}&0\\
0&0&d_{33}\\
\end{array}
\right)$
&
$\left(\begin{array}{cccc}
0&0&0\\
0&d_{22}&0\\
0&0&d_{33}\\
\end{array}
\right)$
&
$\left(\begin{array}{cccc}
d_{11}&0&0\\
0&d_{22}&0\\
0&0&d_{33}\\
\end{array}
\right)$
\\ 
$T^{4}_3$ :&
$\left(\begin{array}{cccc}
d_{11}&d_{12}&d_{13}\\
0&d_{11}&0\\
0&d_{32}&d_{33}\\
\end{array}
\right)$
&
$\left(\begin{array}{cccc}
0&d_{12}&0\\
0&0&0\\
0&d_{32}&0\\
\end{array}
\right)$
&
$\left(\begin{array}{cccc}
d_{11}&d_{12}&d_{13}\\
0&d_{11}&0\\
0&d_{32}&d_{33}\\
\end{array}
\right)$
\\ 
$T^{5}_3$:&
$\left(\begin{array}{cccc}
0&d_{12}&0\\
0&0&0\\
0&0&d_{33}\\
\end{array}
\right)$
&
$\left(\begin{array}{cccc}
0&d_{12}&0\\
0&0&0\\
0&0&0\\
\end{array}
\right)$
&
$\left(\begin{array}{cccc}
0&d_{12}&0\\
0&0&0\\
0&0&d_{33}\\
\end{array}
\right)$
\\ 
$T^{6}_3$ : &
$\left(\begin{array}{cccc}
d_{11}&d_{12}&0\\
0&d_{11}&0\\
0&0&d_{33}\\
\end{array}
\right)$
&
$\left(\begin{array}{cccc}
0&d_{12}&0\\
0&0&0\\
0&0&d_{33}\\
\end{array}
\right)$
&
$\left(\begin{array}{cccc}
d_{11}&d_{12}&0\\
0&d_{11}&0\\
0&0&d_{33}\\
\end{array}
\right)$
\\ 
$T^{7}_3$ : &
$\left(\begin{array}{cccc}
d_{11}&0&0\\
0&d_{11}&0\\
0&0&d_{33}\\
\end{array}
\right)$
&
$\left(\begin{array}{cccc}
0&0&0\\
0&d_{22}&0\\
0&0&d_{33}\\
\end{array}
\right)$
&
$\left(\begin{array}{cccc}
d_{11}&0&0\\
0&d_{11}&0\\
0&0&d_{33}\\
\end{array}
\right)$
\\ 
$T^{8}_3$ : &
$\left(\begin{array}{cccc}
d_{11}&0&0\\
0&d_{22}&0\\
0&0&d_{33}\\
\end{array}
\right)$
&
$\left(\begin{array}{cccc}
0&0&0\\
0&d_{22}&0\\
0&0&d_{33}\\
\end{array}
\right)$
&
$\left(\begin{array}{cccc}
d_{11}&0&0\\
0&d_{22}&0\\
0&0&d_{33}\\
\end{array}
\right)$
\\ 
$T^{9}_3$ : &
$\left(\begin{array}{cccc}
d_{11}&0&0\\
0&d_{11}&0\\
0&0&d_{33}\\
\end{array}
\right)$
&
$\left(\begin{array}{cccc}
0&0&0\\
0&d_{22}&0\\
0&0&d_{33}\\
\end{array}
\right)$
&
$\left(\begin{array}{cccc}
d_{11}&0&0\\
0&d_{11}&0\\
0&0&d_{33}\\
\end{array}
\right)$
\\ 
$T^{10}_3$ : &
$\left(\begin{array}{cccc}
d_{11}&0&0\\
0&d_{11}&0\\
0&0&d_{33}\\
\end{array}
\right)$
&
$\left(\begin{array}{cccc}
d_{11}&0&0\\
0&0&0\\
0&0&d_{33}\\
\end{array}
\right)$
&
$\left(\begin{array}{cccc}
d_{11}&0&0\\
0&d_{11}&0\\
0&0&d_{33}\\
\end{array}
\right)$
\\ 
$T^{11}_3$ : &
$\left(\begin{array}{cccc}
d_{11}&0&0\\
0&d_{11}&0\\
0&0&d_{33}\\
\end{array}
\right)$
&
$\left(\begin{array}{cccc}
d_{11}&0&0\\
0&d_{22}&0\\
0&0&0\\
\end{array}
\right)$
&
$\left(\begin{array}{cccc}
d_{11}&0&0\\
0&d_{11}&0\\
0&0&d_{33}\\
\end{array}
\right)$
\\ 
$T^{12}_3$ : &
$\left(\begin{array}{cccc}
d_{11}&0&0\\
0&d_{11}&0\\
0&0&d_{33}\\
\end{array}
\right)$
&
$\left(\begin{array}{cccc}
d_{11}&0&0\\
0&d_{22}&0\\
0&0&0\\
\end{array}
\right)$
&
$\left(\begin{array}{cccc}
d_{11}&0&0\\
0&d_{11}&0\\
0&0&d_{33}\\
\end{array}
\right)$
\\ 
$T^{13}_3$ : &
$\left(\begin{array}{cccc}
d_{11}&0&0\\
0&d_{11}&0\\
0&0&d_{33}\\
\end{array}
\right)$
&
$\left(\begin{array}{cccc}
d_{11}&0&0\\
0&d_{22}&0\\
0&0&0\\
\end{array}
\right)$
&
$\left(\begin{array}{cccc}
d_{11}&0&0\\
0&d_{11}&0\\
0&0&d_{33}\\
\end{array}
\right)$
\\ 
$T^{14}_3$ : &
$\left(\begin{array}{cccc}
d_{11}&0&0\\
0&0&0\\
0&d_{32}&d_{33}\\
\end{array}
\right)$
&
$\left(\begin{array}{cccc}
d_{11}&0&0\\
0&d_{22}&0\\
0&d_{32}&0\\
\end{array}
\right)$
&
$\left(\begin{array}{cccc}
d_{11}&0&0\\
0&0&0\\
0&d_{32}&d_{33}\\
\end{array}
\right)$
\\ 
$T^{15}_3$ :&
$\left(\begin{array}{cccc}
d_{11}&0&0\\
0&0&0\\
0&d_{32}&d_{33}\\
\end{array}
\right)$
&
$\left(\begin{array}{cccc}
d_{11}&0&0\\
0&d_{22}&0\\
0&d_{32}&0\\
\end{array}
\right)$
&
$\left(\begin{array}{cccc}
d_{11}&0&0\\
0&d_{22}&d_{23}\\
0&d_{32}&d_{33}\\
\end{array}
\right)$
\\ 
$T^{16}_3$ :&
$\left(\begin{array}{cccc}
d_{11}&0&d_{13}\\
0&d_{11}&0\\
0&0&d_{11}\\
\end{array}
\right)$
&
$\left(\begin{array}{cccc}
0&0&0\\
0&0&0\\
0&0&0\\
\end{array}
\right)$
&
$\left(\begin{array}{cccc}
d_{11}&0&d_{13}\\
0&d_{11}&0\\
0&0&d_{11}\\
\end{array}
\right)$
\\ 
$T^{17}_3$ :&
$\left(\begin{array}{cccc}
d_{11}&0&d_{13}\\
0&d_{11}&0\\
0&0&d_{11}\\
\end{array}
\right)$
&
$\left(\begin{array}{cccc}
0&0&0\\
0&0&0\\
0&0&0\\
\end{array}
\right)$
&
$\left(\begin{array}{cccc}
d_{11}&0&d_{13}\\
0&d_{11}&0\\
0&0&d_{11}\\
\end{array}
\right)$
\\ 
$T^{18}_3$ :&
$\left(\begin{array}{cccc}
d_{11}&0&d_{13}\\
0&d_{11}&0\\
0&0&d_{11}\\
\end{array}
\right)$
&
$\left(\begin{array}{cccc}
0&0&0\\
0&0&0\\
0&0&0\\
\end{array}
\right)$
&
$\left(\begin{array}{cccc}
d_{11}&0&d_{13}\\
0&d_{11}&0\\
0&0&d_{11}\\
\end{array}
\right)$
\\ 
$T^{19}_3$ :&
$\left(\begin{array}{cccc}
d_{11}&d_{13}&d_{13}\\
0&d_{11}&0\\
0&d_{32}&d_{33}\\
\end{array}
\right)$
&
$\left(\begin{array}{cccc}
0&d_{12}&0\\
0&0&0\\
0&d_{32}&0\\
\end{array}
\right)$
&
$\left(\begin{array}{cccc}
d_{11}&d_{13}&d_{13}\\
0&d_{11}&0\\
0&d_{32}&d_{33}\\
\end{array}
\right)$
\\ 
$T^{20}_3$ : &
$\left(\begin{array}{cccc}
d_{11}&0&d_{13}\\
0&d_{11}&0\\
0&0&d_{11}\\
\end{array}
\right)$
&
$\left(\begin{array}{cccc}
0&0&0\\
0&0&0\\
0&0&0\\
\end{array}
\right)$
&
$\left(\begin{array}{cccc}
d_{11}&0&d_{13}\\
0&d_{11}&0\\
0&0&d_{11}\\
\end{array}
\right)$
\\ 
$T^{21}_3$ : &
$\left(\begin{array}{cccc}
d_{11}&d_{12}&d_{13}\\
0&d_{11}&0\\
0&d_{32}&d_{33}\\
\end{array}
\right)$
&
$\left(\begin{array}{cccc}
0&d_{12}&0\\
0&0&0\\
0&d_{32}&0\\
\end{array}
\right)$
&
$\left(\begin{array}{cccc}
d_{11}&d_{12}&d_{13}\\
0&d_{11}&0\\
0&d_{32}&d_{33}\\
\end{array}
\right)$
\\ 
$T^{22}_3$ : &
$\left(\begin{array}{cccc}
d_{11}&0&d_{13}\\
0&d_{11}&0\\
0&0&d_{11}\\
\end{array}
\right)$
&
$\left(\begin{array}{cccc}
0&0&0\\
0&0&0\\
0&0&0\\
\end{array}
\right)$
&
$\left(\begin{array}{cccc}
d_{11}&0&d_{13}\\
0&d_{11}&0\\
0&0&d_{11}\\
\end{array}
\right)$
\\ 
$T^{23}_3$ : &
$\left(\begin{array}{cccc}
d_{11}&0&d_{13}\\
0&d_{11}&0\\
0&0&d_{11}\\
\end{array}
\right)$
&
$\left(\begin{array}{cccc}
0&0&0\\
0&0&0\\
0&0&0\\
\end{array}
\right)$
&
$\left(\begin{array}{cccc}
d_{11}&0&0\\
0&d_{11}&0\\
0&0&d_{11}\\
\end{array}
\right)$
\\ 
$T^{24}_3$ : &
$\left(\begin{array}{cccc}
d_{11}&0&0\\
0&d_{11}&0\\
0&0&d_{11}\\
\end{array}
\right)$
&
$\left(\begin{array}{cccc}
0&0&0\\
0&0&0\\
0&0&0\\
\end{array}
\right)$
&
$\left(\begin{array}{cccc}
d_{11}&0&d_{13}\\
0&d_{11}&0\\
0&0&d_{11}\\
\end{array}
\right)$
\\ 
\caption{Generalized $\alpha\beta$-Derivations}
\end{longtable}


\begin{thebibliography}{999}

\bibitem{BGM}Baghiyan, S.; Gallagher, L.; Mainellis, E. ``On Nilpotent Triassociative Algebras" (2023). arXiv:2306.15059

\bibitem{batten mult} Batten, P.; Moneyhun, K.; Stitzinger, E. ``On characterizing nilpotent Lie algebras by their multipliers." \textit{Communications in Algebra}, Vol. 24, No. 14 (1996).

\bibitem{casas} Casas, J. M.; Casado, R.; Khmaladze, E.; Ladra, M. ``More on crossed modules in Lie, Leibniz, associative and diassociative algebras." \textit{Journal of Algebra and Its Applications}, Vol. 16, No. 06 (2017).

\bibitem{edal mult new} Edalatzadeh, B.; Hosseini, S. N. ``Characterizing nilpotent Leibniz algebras by a new bound on their second homologies." \textit{Journal of Algebra}, Vol. 511 (2018).

\bibitem{Cal} Calderón, A.; Sánchez, J. \textit{Journal of Geometry and Physics}, Vol. 110 (2016).

\bibitem{Cheng} Cheng, Y.; Qi, H. ``Representations of Bihom-Lie Algebras." \textit{Algebra Colloquium}, Vol. 29, No. 01 (2022).

\bibitem{D} Das, A. ``Cohomology of BiHom-associative algebras." \textit{Journal of Algebra and Its Applications}, Vol. 21, No. 01 (2022).

\bibitem{Lar} Laraiedh, I.; Silvestrov, S. ``Constructions of BiHom-X algebras and bimodules of some BiHom-dialgebras." \textit{Algebra and Discrete Mathematics}, Vol. 34, No. 2 (2022).

\bibitem{L} Loday, J.-L.; Ronco, M. ``Trialgebras and families of polytopes" (2002). arXiv:math/0205043

\bibitem{dialgebras} Loday, J.-L. ``Dialgebras" in \textit{Dialgebras and related operads}, pp. 7-66. Lecture Notes in Mathematics, Vol. 1763. Springer-Verlag Berlin Heidelberg (2001).

\bibitem{ME} Mainellis, E. ``Multipliers and covers of perfect diassociative algebras." \textit{Journal of Algebra and Its Applications} (2023). https://doi.org/10.1142/S021949882450244X

\bibitem{mainellis di} Mainellis, E. ``Multipliers and Unicentral Diassociative Algebras." \textit{Journal of Algebra and Its Applications}, Vol. 22, No. 05 (2023).

\bibitem{Mainellis nilp} Mainellis, E. ``Multipliers of Nilpotent Diassociative Algebras." \textit{Results in Mathematics}, Vol. 77, No. 191 (2022).

\bibitem{Mainellis} Mainellis, E. ``Nonabelian extensions and factor systems for the algebras of Loday." \textit{Communications in Algebra}, Vol. 49, No. 12 (2021).

\bibitem{Mak} Makhlouf, A.; Zahari, A. ``Structure and Classification of Hom-Associative Algebras." \textit{Acta et commentationes universitis Tartuensis de mathematica}, Vol. 24, No. 1 (2020).

\bibitem{BZI} Mosbahi, B.; Zahari, A.; Basdouri, I. ''Classification, $\alpha$-Inner Derivations and $\alpha$-Centroids of Finite-Dimensional Complex Hom-Trialgebras." \textit{Pure and Applied Mathematics Journal}, Vol. 12, No. 5 (2023).

\bibitem{Rak} Rakhimov, I.S. ``On Central Extensions of Associative Dialgebras." \textit{J. Phys.: Conf. Ser.}, Vol. 697, No. 012009 (2016).

\bibitem{Y} Yau, D. ``(Co)homology of triassociative algebras." \textit{International Journal of Mathematics and Mathematical Sciences} (2006).

\bibitem{Zah} Zahari, A.; Bakayoko, I. ``On BiHom-Associative dialgebras." \textit{Open J. Math. Sci.}, Vol. 7 (2023).

\bibitem{Zah2} Zahari, A.; Mosbahi, B.; Basdouri, I. ``Classification, Derivations and Centroids of Low-Dimensional Complex BiHom-Trialgebras" (2023). arXiv:2304.06781

\end{thebibliography}
\end{document}